\documentclass[leqno,12pt]{article} 
\usepackage{amsmath, amssymb}

\usepackage{amsthm} 
\newcommand{\ol}{\overline}
\def\an{\mathrm{an}}

\def\N{{{\mathbb N}}}
\def\A{{{\mathbb A}}}

\def\Z{{{\mathbb Z}}}
\def\Q{{{\mathbb Q}}}
\def\R{{{\mathbb R}}}
\def\C{{{\mathbb C}}}
\def\P{{{\mathbb P}}}
\def\G{{{\mathbb G}}}

\newcommand{\Qbar}{\overline{\Q}}

\def\H{\mathcal{H}}
\def\V{\mathcal{V}}

\def\E{{\mathcal E}}
\def\Ex{{\mathcal Ex}}
\def\F{{\mathcal F}}
\def\T{{\mathcal T}}

\def\1C{{{\mathcal C}}}
\def\X{{{\mathcal X}}}

\def \z {\mathbf{z}}

\newcommand{\bA}{\mathbb{A}}

\newcommand{\bG}{\mathbb{G}}

\newcommand{\cA}{\mathcal{A}}

\newcommand{\cO}{\mathcal{O}}

\newcommand{\cF}{\mathcal{F}}

\newcommand{\ra}{\rightarrow}

\def\v{{\bf v}}
\def\e{{\bf e}}

\theoremstyle{plain} 
\newtheorem{theorem}{\indent\sc Theorem}[section]
\newtheorem{lemma}[theorem]{\indent\sc Lemma}
\newtheorem{corollary}[theorem]{\indent\sc Corollary}

\theoremstyle{definition} 
\newtheorem{definition}[theorem]{\indent\sc Definition}
\newtheorem{remark}[theorem]{\indent\sc Remark}

\newtheorem{conj}{Conjecture}[section]
\newtheorem{thm}{Theorem}[section]

\makeatletter
\def\address#1#2{\begingroup
\noindent\parbox[t]{7.8cm}{%
\small{\scshape\ignorespaces#1}\par\vskip1ex
\noindent\small{\itshape E-mail address}%
\/: #2\par\vskip4ex}\hfill%
\endgroup}%
\makeatother
\author{Pietro Corvaja, Jacob Tsimerman, Umberto Zannier}
\title{\uppercase{Finite Orbits in Surfaces with a Double Elliptic Fibration and torsion values of sections}}
\date{\today}
\begin{document}

\maketitle


\begin{abstract}
 {We consider surfaces with a double elliptic fibration, with two sections. We study the orbits under the induced translation automorphisms proving that, under natural conditions, the finite orbits are confined to a curve. This goes in a similar direction of (and is motivated by) recent work by Cantat-Dujardin \cite{CD}, although we use very different methods and obtain related but different results.
 
 As a sample of application of similar arguments,  we prove a new case  of the Zilber-Pink conjecture, namely Theorem \ref{T.new1}, for certain schemes over a 2-dimensional base, which was known to  lead to  substantial difficulties. 
 
Most results rely, among other things,  on recent theorems by Bakker and the second author of `Ax-Schanuel type'; we also relate a functional condition with a theorem of Shioda on unramified sections of the Legendre scheme. For one of our proofs, we   also use recent height inequalities by Yuan-Zhang \cite{YZ} (or those by Dimitrov-Gao-Habegger \cite{DimGaoHabegger}) .  

Finally, in an appendix, we show that the Relative Manin-Mumford Conjecture over the complex number field is equivalent to its version over the field of algebraic numbers.}
\end{abstract}

\section{Introduction} \label{S.Int}

This paper originally arose through work of S. Cantat and R. Dujardin, who,  in the paper \cite{CD}, study {\it non elementary} subgroups of automorphisms of a  projective surface, and especially the orbits. They prove various results which confine the  finite orbits to a finite union of curves, under appropriate conditions. 

\medskip

{\bf Double elliptic fibration}. Significant  cases arise when the surface $\X$ is endowed with a
 double elliptic fibration;  namely, we have  two rational maps $\lambda,\mu:\X\to \P_1$,  such that for each of them the general fiber is an elliptic curve  on $\X$, with a zero section.\footnote{Note that allowing one of the base curves to be of positive genus would imply the other scheme  to be isotrivial, since the existence of a non constant morphism from a genus one curve to a positive genus curve implies that the genus one curve is a factor of the jacobian of the target curve.} We assume that the morphism $(\lambda,\mu): \X \to \P_1^2$ has finite degree, which amounts to saying that the two elliptic fibrations are distinct. 
  They correspond to two elliptic schemes over a rational  base, denoted $\E,\F$ respectively, with fibers denoted $E_a, F_b$. 
  
  We assume throughout for simplicity that none of these schemes is isotrivial. 
  
  We can also assume that $\X$ is smooth and projective.

\medskip

We further suppose to have two sections  $\sigma,\tau$, resp.  for $\E,\F$, none of which is torsion.

These sections define rational automorphisms $\tilde\sigma,\tilde\tau$ of $\X$, e.g. putting
\begin{equation}
\tilde\sigma(p):=p+\sigma(\lambda(p)),\qquad  p\in \X,
\end{equation}
where the addition is meant on $E_{\lambda(p)}$, and similarly for $\tau$ in place of $\sigma$.

This situation is a key example of the context of \cite{CD}, where as mentioned above the authors prove that all the points outside a suitable curve have infinite orbit under the relevant group, which in the present case would be the one generated by  $\tilde{\sigma},\tilde{\tau}$.  

{\bf Goals}. One of the goals of this paper is to prove more precise conclusions in these  cases, using moreover a completely different method. Specifically, we shall show that already the orbit of a point under   a `small' part of the group is infinite, if the point lies outside a prescribed finite union of curves.


\medskip

{\bf An example}. A nice well-known example of the  situation is provided by the quartic Fermat surface in $\P_3$ given by $x^4+y^4=z^4+w^4$. (See e.g. the paper \cite{CZ2} for a study of arithmetic properties of that surface exploiting the presence of a double fibration.)  Swinnerton-Dyer already constructed elliptic fibrations with sections, as follows: for a line on $\X$ (there are 48 of them), consider the pencil of planes through it, which intersects $\X$ in the line plus a cubic curve in the plane. By considering another line, its intersection with the cubic provides a section, which can be taken to be the zero section, so we have an elliptic family. Still another line gives another section.

The automorphisms which arise in this situation may be shown  to extend regularly to $\X$.    However our method (with appropriate definitions) works even for rational automorphisms which cannot be extended to regular ones. 

\medskip

{\bf Higher dimensions ?} In this paper we consider only the case of surfaces,   motivated by the mentioned work of Cantat-Dujardin. However, in principle the methods would apply also to analogous situations in higher dimensions, for instance to the case of $n$ elliptic fibrations on a variety of dimension $n\geq 3$. These can be easily constructed, for instance by taking three 2-nets  of elliptic curves in $\P_3$.  
 
But in  dimension $\ge 3$ there are other possibilities; for instance  in dimension $3$  we could consider both elliptic fibrations over a surface and  fibrations over a curve by  abelian surfaces (possibly in the birational sense). We believe that the method of this paper may be applied to such general context, using the recent results of Yuan-Zhang on heights in place of the Silverman-Tate bounds. We leave this problem for future work and we limit to the present case,  that represents our original motivation.

\medskip

Going back to our situation, 
 we denote by $\Gamma$ the group generated by $\tilde\sigma,\tilde\tau$, and we are interested in the set of points of $\X$ having {\it finite} orbit through $\Gamma$. 
      We  assume from now on that all of these objects are defined over a number field $K$.

\medskip

{\bf About the automorphisms}. 
Generally, these automorphisms are not defined everywhere, but  each of them is  well-defined outside a finite set. Also, they might contract some (rational) curve to a points; this last fact is due to the fact that one cannot in general suppose that both elliptic fibrations are relatively minimal.  In view of this possible phenomenon, we avoid points of indeterminacy by giving the definitions which follow.

For a point $p\in \X$ we set 
\begin{equation}\label{E.orbit}
O(p)=\{\tilde\sigma^m\tilde\tau^n(p): \quad m,n\in\N,\quad  \hbox{$\tilde\sigma^m\tilde\tau^n$ is defined at $p$}\}.
\end{equation}

Our attention is on {\bf  the set of $p$ such that $O(p)$ is finite}, 
for which we shall  prove the following 

\begin{theorem} \label{T.main}  The points $p\in \X$ such that $O(p)$ is finite lie in a finite union of curves of type $F_x$.
\end{theorem}

Under mild assumptions, the theorem implies actually  the finiteness of the set of points whose $\Gamma$-orbit is finite.  For instance, more precisely, we can deduce the

\begin{corollary}\label{C.main}  Suppose that the map $(\lambda,\mu):\X\to\P_1^2$ is finite. 
Then there are only finitely many points $p\in \X$ such that all sets $O(\tilde\sigma^h(p))$, $h\in\N$, are finite. 
\end{corollary}

\begin{proof} [Deduction of the corollary from the theorem.] Note that indeed the corollary is an immediate consequence of the theorem,   applied with $\tilde\sigma^h(p)$ in place of $p$. To prove this, note that  if $p$ is one of the points relevant in the corollary, all the $\tilde\sigma^h(p)$, $h\in\N$, lie in a certain finite union  $U=\bigcup_{x\in A}F_x$, where $A$ is a finite subset of $\P_1$. Now, since the fibrations are distinct by assumption,   for any $x$  only finitely many  curves $\tilde\sigma^h(F_x)$ can be  of the shape $F_y$: otherwise infinitely many fibers of $\mu$ would be isomorphic and the $\mu$-fibration would be isotrivial, which is excluded by the assumptions.  

Suppose that our set is infinite, then its intersection with some component $F'$ of an $F_a$ would be infinite, for suitable $a\in A$. Then for each $h\in\N$ the curve $\tilde\sigma^h(F')$ would be inside $U$. But then some nonzero power $\tilde\sigma^q$ would stabilise $F'$. On the other hand, by the assumption the function $\lambda:F'\to \P_1$ is generically surjective, implying that $F'$ contains a point $p'$ such that $\tilde\sigma^q$ has infinite order at $p'$. But this orbit lies in the intersection of $F'$ with the $\lambda$-fiber at $p'$, and this intersection is finite by assumption. The contradiction proves the assertion.
\end{proof}

 Actually, this argument shows that for the same conclusion it suffices to let $h$ run through a sufficiently large set of integers (rather than the whole $\N$).

Also, the corollary asserts that

{\it  There are only finitely many  points $p\in \X$ such that the set $\{\tilde\sigma^m\tilde \tau^n\tilde \sigma^h(p), m,n,h\in\Z\}$ is finite. }

\medskip 
 
 \subsection{A kind of  improvement}
 We can formulate the above situation in the following terms: we have a scheme over an open  subset $\X_0$ of $\X$ with fibers given as products of three elliptic curves. Namely, above a point $x\in\X_0$ we associate the product of elliptic curves
$$
E_{\lambda(x)}\times F_{\mu(\tilde\sigma(x))}\times E_{\lambda(\tilde\tau\circ\tilde\sigma(x)))}.
$$
Here $X_0\subset X$ is obtained by removing  from $\X$ the subset where some of the relevant maps is not defined, or where some elliptic curve degenerates; hence $X_0$ is the complement of finitely many curves and points in $X$.

 For simplicity of notation, we denote by $\E_i\to \X_0$, $i=1,2,3$, the corresponding elliptic schemes.

We have an obvious section  $\xi$ given by  
$$
\xi(p)=\sigma(p)\times \tau(\tilde\sigma(p))\times \sigma((\tilde\tau\circ\tilde\sigma)(p)),
$$
where we simplified the notation by writing $\sigma(p)$ for $\sigma(\lambda(p))$, $\tau(p)$ for $\tau(\mu(p))$.

Using a recent result of Bakker with the second author, we shall prove a variation on Theorem \ref{T.main}, which is stronger in a sense.

We start by noting that the orbit $O(p)$ of a point $p$ is certainly infinite if either $\sigma(p)$ or $\tau(p)$ has infinite order on the relevant elliptic curve. Then it makes sense to study the points $p$ where both sections have finite order.
\medskip

\begin{definition}
 We define the sets $T_\lambda,T_\mu$ by setting  
\begin{equation*}
T_\lambda=\{x\in \X_0: E_{\lambda(x)} 
\sigma( x)\ \hbox{has finite order on $E_{\lambda(x)}$}\}
\end{equation*}
 and similarly for $\tau,\mu$.  
 \end{definition}

Each of these sets is known to be a  countable infinite union of curves (this is not obvious, but may be proved e.g. via the non-constancy of the so-called Betti map, or see \cite{CDMZ}, \cite{Z}). The intersection  $T_\lambda\cap T_\mu$ is then expected to be a countable Zariski-dense set. Similarly, the set of points $x\in T_\lambda$ such that $\tilde{\sigma}(x)\in T_\mu$ is again Zariski-dense. 

We want to show that if we add a third condition we end up with a degeneracy conclusion. With this in mind we define the exceptional set

\begin{equation}
\Ex:=\{x\in \X :  x\in T_\lambda,\  \tilde\sigma(x)\in T_\mu,\ \tilde\tau\circ\tilde\sigma(x)\in T_\lambda\}.
\end{equation}

\medskip

\begin{theorem}\label{T.BT} 
The set $\Ex$ is contained in a finite union of curves.  
\end{theorem}

Each of these results somewhat improves on the conclusions of \cite{CD} {\it for the special case } of groups considered here.\footnote{The paper \cite{CD} on the other hand explores a much more general situation. It uses completely different methods.}







\medskip

\subsection{On the Betti map and Relative Manin-Mumford Conjecture on a two-dimensional base} While conceiving this work, we realized  a few links with other problems. 

(1) {\bf About Ax-Schanuel and the Betti map}. The applications of known arithmetical methods led us to study new problems concerning the so-called Betti maps, related in turn to functional transcendence and theorems of Ax-Schanuel type; in particular we raised the question solved in the paper \cite{BT}.

(2) {\bf Relative Manin-Mumford}. Some of the methods lead to cases of the Zilber-Pink conjectures for torsion points which until very recently  weren't yet answered. Here we are thinking of the part of Zilber-Pink conjecture concerning the so-called {\it relative Manin-Mumford problem}: to confine the torsion values of a section of an abelian scheme to the appropriate closed subvariety. This had been dealt essentially for abelian schemes over a curve (see e.g. \cite{CMZ} for a recent result and a discussion of previous works).  In \cite{HabAnnali} Ph. Habegger treated, via a somehwat {\it ad hoc} method, an instance of an elliptic scheme over a surface. 
 Recently Gao and Habegger have announced a complete solution over varieties of arbitrary dimension. The methods of the present work, relying on the Bakker-Tsimerman results on Ax-Schanuel type,  combined with results of Yuan-Zhang on heights, should allow a partially different proof of their results. Here we include only a sample of this which in fact was present in a previous version of this paper, predating Gao-Habegger announcement.

 More specifically we shall prove the following

 \begin{theorem}\label{T.new1}:  Let $\E_1,\E_2,\E_3$ be pairwise non-isogenous elliptic schemes over a surface $\X$, and let $\sigma_1,\sigma_2,\sigma_3$ be non torsion sections  (all defined over $\overline\Q$). Then there are only finitely many $x\in \X$ such that $\sigma_i(x)$ is torsion on $\E_i$ for $i=1,2,3$.
 \end{theorem}



To prove this theorem, beyond the mentioned result of Bakker-Tsimerman, we shall use a recent result of Yuan-Zhang on height inequalities for values of sections.

The previous results could be obtained as special cases of this. However, we have preferred to develop independent proofs, free of these tools, whose arguments could be useful for other tasks.

\medskip

(3) {\bf Relative Manin-Mumford over $\C$.} While Gao-Habbeger have announced their proof for varieties over $\Qbar$, we realized that the transcendence methods invoked in this paper allow  a specialization argument to formally deduce the general case over arbitrary complex bases. In the appendix, we give a proof of this deduction. \footnote{Ph. Habegger  informed us that he and Gao also developped a reduction from $\Qbar$ to $\C$, somewhat different from ours.} 

\medskip

 (4) {\bf Ramification of sections of the Legendre scheme}. As usual in the applications of the Pila-Wilkie counting, one has to deal eventually with the so-called {\it algebraic part } of the relevant varieties, proving that this is empty(or small). In turn, this leads to questions of {\it functional type}. In our case, the treatment of the algebraic part for Theorem \ref{T.BT} led us to a statement seemingly of independent interest, for doubly elliptic schemes over a curve. Namely, we prove
 the following result, which seems not free of independent interest, and that we couldn't locate in the literature.

\begin{theorem} \label{T.shioda} Let $\E$ be a non-isotrivial elliptic scheme over a complex algebraic curve $B$, with a non-torsion section $\sigma$ and consider an irreducible algebraic  curve $T$ inside $B\times B$ with the following property: for generic $(x,y)\in T$, there is a nonzero cyclic isogeny $\phi:E_x\to E_y$ such that the sections on $T$ given by $\phi\circ \sigma$ and $\sigma$ are linearly dependent, i.e. there are nonzero integers $a,b$ such that $a\phi(\sigma)(x)=b\sigma(y)$. Then 
$\phi$  is an isomorphism. 
\end{theorem}

To prove this we  found it necessary to invoke a result by Shioda on ramification of fields of definition for a non-torsion section of the Legendre scheme.  We shall say more on this below.

\medskip

\medskip

\section{Proofs} 

In the sequel we  fix a number field $K$ where the variety, the fibrations, sections are defined and such that all the indeterminacy points for $\tilde{\sigma}, \tilde{\tau}$ are $K$-rational.  

We start by proving Theorem \ref{T.main}. This proof will follow an application of Pila-Wilkie estimates for rational points on transcendental varieties, introduced in \cite{MZ}.
However this is not automatic because we work now over a surface, and an additional variable appears. It turns out that the proof can be carried out with sufficient uniformity so as to manage with this freedom. We believe this may be useful in future applications, and this is one more motivation for inserting this argument, and not merely appealing to the subsquent arguments and results.

\medskip

We recall the following lemma, which follows immediately from results proved first by Silverman and Tate, actually in much greater generality. We shall later need a stronger result recently proved by Yuan-Zhang in this context; a similar result, sufficient for our purposes, was obtained by Dimitrov-Gao-Habegger in \cite{DimGaoHabegger}.

We denote by $h(x)$ the Weil height of the algebraic number (or point) $x$.

\begin{lemma} \label{L.silverman} There is a number $C$ such  that  for each $x\in \P_1$ such that   $\sigma(x)$  has finite order  on $E_x$ (or $\tau(x)$ has finite order on $F_x$), we have  $h(x)<C$. 
\end{lemma}

Note that automatically these points are algebraic (since the sections are supposed to be non torsion). 
For the bound of the lemma we may apply the quoted results because we are assuming that the elliptic schemes are not isotrivial. (See \cite{Z}, Prop. 3.2, p. 69, for a simple self-contained argument.) 

The boundedness of the  height will be used in several respects, including the possibility of confining our relevant points to prescribed compact sets (as in the paper \cite{MZ2} - especially Lemma 8.2 - and subsequent ones). 

\medskip

Another important tool will be the following result:

\begin{lemma}\label{L.masser} There is a constant $c=c(\X,\sigma,\tau)$ such that for any $x\in \P_1$ such that $\sigma(x)$ (resp. $\tau(x)$) is torsion, of exact order $N$  on $E_x$ (resp. $F_x$), then 
$$N\le c[\Q(x):\Q]^2(1+h(x))$$.
\end{lemma}

This version (appearing as Lemma 7.1 of \cite{MZ2}) is due S. David. A similar but somewhat weaker result (which would be still sufficient here) was due to Masser; see Appendix D (by Masser) to \cite{Z} for a sketch of proof.

\medskip

Let now $p\in \X_0$  such that $O(p)$ is finite. Naturally $p$ is an algebraic point, and we denote  $a=\lambda(p), b=\mu(p)\in\overline\Q$, so that $p\in E_a\cap F_b$.

\medskip

We may now assume that $\tau(b)$ has finite order $m$ on $F_b$, and that the point $\tilde \tau^r(p)$, $r=0,1,\ldots ,m-1$, has finite order $n_r$ on $E_a$. 

\medskip

If we have a bound on $m$ 
 then  $b$ lies in a prescribed finite set depending on the bound, since $\tau$ is not a torsion section; hence $p$ belongs to a finite union of fibers for $\tau$. 
 Therefore, for proving Theorem \ref{T.main},  in what follows we may assume that 
 $m$ is larger than any prescribed number.  

\medskip

We now observe that the curve $F_b^0:=F_b\cap \X_0$ may be seen as the base curve for two elliptic schemes with sections. Namely, to $z\in F_b^0$ we associate the curves $E_{\lambda(z)}$, $E_{\lambda(z+\tau(b))}$ and the two respective sections given by 
$$
s_0(z):=\sigma(\lambda(z))\quad  \mathrm{ and} \quad s_1(z):=\sigma(\lambda(z+\tau(b))).
$$
  This yields a product scheme over $F_b^0$ with a section given by $(s_0,s_1)$.

\medskip

We note two things:

1. None of these two sections is (identically)  torsion. This holds by assumption, for otherwise $\sigma$ itself would be identically torsion.  

2. Also, for large enough $m$,  the two elliptic schemes in question are non-isogenous. In fact otherwise the locus of bad reduction would be the same, and hence would be preserved under  translation by $\tau(b)$. In particular since this locus is non-empty (because the schemes are non isotrivial), the cardinality of the bad locus would be at least $m$. For large $m$ this cannot hold (since the degree of $\lambda$ as a map on $F_b$ is bounded), proving our contention. 

\bigskip

Before going ahead, note that by Galois conjugation we may replace throughout $b$ (and $p$) by any of its conjugates over $K$. Since $b$ has bounded height, we may then suppose that $b$ lies in a sufficiently large prescribed compact set $\Delta$ of $\P_1(\C)$, not containing any of the finitely many points of bad reduction for $\mu$ or $\lambda$ or $\lambda\circ\tilde{\tau}$  (and actually we might even  assume that $\Delta$ contains at least half of the conjugates of $b$  over $K$, which at the moment is not needed here,  
 see  Lemma 8.2 of \cite{MZ2} for all of  this). 

\medskip

At this stage we subdivide the proof in two parts. The first part will roughly represent a {\it uniform} version of the results in \cite{MZ2}, and we shall follow that paper, indicating briefly the various steps and the necessary modifications.

Recall that we are assuming that $m$, i.e. the torsion order of $\tau(b)$, is large enough (to justify the steps to follow).

\medskip

To define our  separation  into cases, for $r=0,1,\ldots ,m-1$, we let  $n_r$ denote  the order of $\sigma$ at the point $\lambda(p+r\tau(b))$: we can indeed assume that this value of $\sigma$ is of finite order on $E_{\lambda(p+r\tau(b))}$, for otherwise $O(p)$ is plainly infinite.

\medskip

\centerline{{\bf FIRST CASE}: We have  {\it $n_r>m^5$ for 
some $r$.}}

\medskip

For $b$ in the compact set $\Delta$, we  first remove from $F_b(\C)$ a very small open  neighbourhood of the set of bad reduction relative to $\lambda(z)$ and $\lambda(z+\tau(b))$. We may do this smoothly as $b$ varies in $\Delta$. (This is the same as removing small neighbourhoods of the singular fibers of $\lambda$ and of $\lambda\circ \tilde\tau$, restricted to $F_b$ for  $b\in\Delta$.)  Actually, we may use real-algebraic functions to express these neighbourhoods; the relevance of this is that these functions are {\it definable} (in the {\it $o$-minimal structure $\R_{an,exp}$}) and will allow the application of results by Pila-Wilkie. (For rudiments of these notions see also \cite{Z}, where moreover sketches of proofs relevant here are given.)  This is in practice like saying that we may assume $\Delta$ to be definable.

Then we cover the complementary compact set in $F_b(\C)$ with finitely many small compact simply-connected sets (again definable) such that for $z$ inside any of them we may express analytically periods $\omega_{1i},\omega_{2i}$, $i=0,1$,  for the complex tori corresponding to the curves $E_{\lambda(z+i\tau(b))}$, and  elliptic logarithms $\ell_i(z)$ of the section.

Note that all these functions  depend also on $b$ (a fact which we do not express explicitly for  notational convenience), but since $b$ lies in a compact set we may further suppose that these expressions hold uniformly and analytically for $b$ in a small disk $\Delta_1\subset\Delta$. 

We may further write, for $z$ varying on each of these finitely many small compact sets inside $F_b$, 
\begin{equation}\label{E.betti}
\ell_i(z)=\beta_{1i}(z)\omega_{1i}(z)+\beta_{2i}(z)\omega_{2i}(z),\qquad i=0,1,
\end{equation}
for real analytic functions $\beta_{1i},\beta_{2i}$ (depending also analytically on $b\in\Delta_1$).
Note that $s_i(z)$ is torsion (on $E_{\lambda(z+i\tau(b))}$) if and only of $\beta_{1i}(z)$, $\beta_{2i}(z)$ are rational numbers (with common denominator equal to the order of $s_i(z)$).

 Let us pick one of the compact sets in question, containing  one 
value of $z$, denoted $\zeta$,  with the property that  both $s_0(\zeta), s_1(\zeta)$ are torsion (in the respective curve) with   maximum order  $n>m^5$.  (This $\zeta$ shall be of the shape $p+r\tau(b)=\tilde\tau^r(p)$, $0\leq r\leq m-1$) 

As above, by Lemma 8.2 of \cite{MZ2}, since $\zeta$ has bounded height, we may assume that the compact set does not intersect a sufficiently small open set, given in advance, containing all the finitely many points of bad reduction for $\lambda$.  In fact, otherwise, since the number of compact sets of our covering is prescribed since the beginning, a positive percentage of the conjugates of $\zeta$ would lie near a prescribed algebraic point of $\X_0$, forcing the height of $\zeta$ to be larger than any prescribed number (if $m$ is large enough). 

\medskip


\medskip

By Lemma \ref{L.masser}, since $h(\zeta)$ is bounded (by the previous Lemma \ref{L.silverman}), we get $m^5<n\le c[K(\zeta):K]^2$.  On the other hand the usual division equations yield $[K(b):K]\ll m^2\ll n^{2/5}$. Hence $[K(b,\zeta):K(b)\gg n^{1/2-2/5}=n^{1/10}$. 

We can take then conjugates $\zeta_j$, $j=1,\ldots ,n_1\gg n^{1/10}$ of $\zeta$ over $K(b)$, and in this way we obtain simultaneous  torsion values $s_i(\zeta_j)$ on $E_{\lambda(z+i\tau(b))}$ ($i=0,1$) for all $j=1,\ldots ,n_1$. 

Note that each $\zeta_j$ has uniformly bounded height by Lemma \ref{L.silverman} and again by Lemma 8.2 of \cite{MZ2}  we may assume that at least $c_1n_1$ of these conjugates lie in a same compact set,  denoted now $\Omega=\Omega_b\subset F_b(\C)$, with the same properties as above, i.e. that the above functions are well-defined and analytic in $\Omega$: here $c_1>0$ denotes a positive number depending only on the original data. Hence we may assume that $\Omega$ contains $\zeta_j$ for $j=1,\ldots, l$ where $l\gg n^{1/10}$. 

\medskip

Each  $\zeta_j$, $1\le j\le l$,  gives rise to a  rational point $B(\zeta_j)$,  of denominator  $\ge n>m^5$, on the real-analytic surface  described in $\R^4$  by $B(z)=B_b(z):=(\beta_{10}(z),\beta_{20}(z), \beta_{11}(z),\beta_{21}(z))$ as $z$ varies in $\Omega$.  We denote by $Z_b$ this surface. 

We further note that  each rational point $\rho\in\Q^4$ appears as a value $B(\zeta_j)$ for at most $C_2$ points $\zeta_j$, where $C_2$ depends only on the opening data: this is because a function 
$\ell_i(z)-\rho_1\omega_{1i}(z)-\rho_2\omega_{2i}(z)$   has a uniformly bounded number of zeros in $\Omega$ for given $\rho_1,\rho_2$. In turn,  this follows as in Lemma 7.1. of \cite{MZ}, the only difference is that here $b$ (which does not even appear in the notation) may vary; but locally the lemma holds uniformly, with the same proof. (Alternatively, one may use directly Gabrielov's theorem, see \cite{BM}.) 

\medskip

We conclude that the real-analytic surface $Z_b$ in $\R^4$ given by  $\{B(z): z\in\Omega\}$ contains at least $l$ rational points of denominator $\le n^2$, where, as above $l\gg n^{1/10}$. These points will have height which is $\ll n^2$ where the implicit constant depends only on the involved compact sets, which are supposed to be fixed from the beginning.

At this point we apply Theorem 1.9 of Pila-Wilkie's paper \cite{PW}, to obtain a bound for the number of these rational points in the {\it transcendental part} of $Z_b$ (see \cite{PW} or \cite{Z} for definitions). The difference from the present argument compared to \cite{MZ2} is that here $b$ is varying. More precisely, we may first  decompose $\Delta$ as above, as a finite union of small definable compact sets $A_j$ so that the relevant functions are analytic in the union $Z:=\bigcup_{b\in A_j}b\times \Omega_b\subset  A_j\times\R^4$. This $Z$ is a definable family, where $Z_b$ are the fibres. 

From Theorem  1.9 of  \cite{PW} we obtain that the transcendental part of $Z_b$ contains at most $c(Z,\epsilon)T^\epsilon$ rational points of height $\le T$, for any $\epsilon >0$.

\medskip

 On the other hand, the algebraic part of $Z_b$ is empty: this follows from Andr\'e's results in \cite{A}, applied  as in \cite{MZ2}. Let us briefly recall the steps of this deduction. If the algebraic part is non-empty there is a real-analytic arc $U$ contained in $Z_b$. This means that the four  functions $\beta_{ij}$ restricted to this arc all depend algebraically on any of them which is nonconstant (if all are constant the thing is even easier).  But then the two functions $\ell_j$,  when restricted to $U$, are algebraically dependent on the restrictions of the $\omega_{ij}$, $i=1,2$, $j=0,1$ to this real algebraic arc. However all  these functions are complex analytic (locally), so the dependence would hold identically on their domain, which violates the said result by Andr\'e \footnote{Actually, while Andr\'e proved a more general result, here we might even apply a previous result by D. Bertrand}. 

\medskip

At this point, taking $T=c_3n^2$ and taking $\epsilon=1/10$, say, we obtain a bound for $n$  putting together  the above estimates. 

This concludes the treatment of the first case.

\medskip

We now go to the

\medskip 

\centerline{{\bf SECOND CASE}: We have  {\it $n_r\le m^5$ for all  $r$.}}

\medskip

This means that the torsion orders of $\sigma$ at the various $\lambda(\tilde\tau^r(p))$ are not too much larger than the torsion order $m$ of $\tau$ at $\mu(p)=b$. 

The strategy will be similar to the one for  the first part, but the rational points will be constructed somewhat differently, and the proof of emptyness of the algebraic part will also be a bit different, not using the mentioned result by Andr\'e but using a self-contained argument.

\medskip

Like for the first part, we suppose that $b$ lies in the compact set $\Delta$ (not containing points of bad reduction for $\mu$). We may actually include $\Delta$ in a slightly larger compact set $\Delta'$ with the same properties, so to be able to work safely on the boundary of $\Delta$. Now, locally on $\Delta$ we have a torus representation for the curves $F_\eta$, $\eta\in\Delta$; hence we may cover $\Delta$ with finitely many open simply connected sets $A_l$ contained in $\Delta'$, such that for $\eta$ in one of these $A_l$ the torus corresponding to $F_\eta$ has periods $\pi_1(\eta),\pi_2(\eta)$ varying analytically on $A_l$.  We may also assume that these $A_i$ are definable.

For our purposes we may work on each $A_l$ at a time, and hence let us now omit the subscripts and  denote by $A$ one of these open sets containing $b$.

We denote by $\xi$ a complex variable lying in a (compact) fundamental domain $T(\eta):=\{t_1\pi_1(\eta)+t_2\pi_2(\eta): 0\le t_1,t_2\le 1\}$ for the torus corresponding to $F_\eta$: note that these fundamental domains, for $\eta\in A$, form a definable family. Let us now remove from each fundamental domain small open disks with centers corresponding to  the points of bad reduction for $\lambda$ restricted to the elliptic curve $F_\eta$. The remaining domain, denoted $T^*(\eta)$,  will continue to form a definable family for suitable parametrisations of the boundaries of these disks. 

Note that, as above, by bounded height of $a,b$ and $\lambda(p+r\tau(b))$ (obtained from Lemma \ref{L.silverman}), we may assume that our points have toric representatives lying in  $T^*(b)$. Let us be a bit more precise here. 
 For each $l$, and for $\eta\in A_l$, we remove, as said above,  from $T(\eta)$ very small disks around the points of bad reduction for $\lambda$ (viewed as a function on $T(\eta)$ through the isomorphism $T(\eta)\cong F_\eta$). We do this choice of disks  in an algebraic way (so, definable way) as $\eta$ varies in $A_l$. Let now $K_1=K(p)$, of degree $d_1$ over $K$, and let us consider the conjugates $p^g$ of $p$ over $K$. By bounded height of $b$, we may assume that at least $\gg d_1$ conjugates of $b$ \footnote{Some of these conjugates may coincide: we are taking conjugates of the field $K_1$, which may be larger than $K(b)$.} will lie in a same $A$ among the $A_l$.  (Indeed, the number of $A_l$ is fixed, while $\gg d_1$ conjugates stay in the compact union of the $A_l$.) 
 
 If for all of these conjugates there exist $m/2$ (say) values of $r$ such that $\lambda(p^g+r\tau(b^g))$ lies in someone of the small disks to be removed, then the sum $\sum_g \sum_{r_1\neq r_2}|\lambda(p^g+r_1\tau(b^g))-\lambda(p^g+r_2\tau(b^g))|^{-1}$, made over all conjugations  $g$, will exceed any given multiple of $d_1m^2$ (if the said disks are small enough), contradicting uniformly bounded height of $\lambda(p+r\tau(b))$. (This argument is a refinement of that for Lemma 8.2 of \cite{MZ2}.)   Hence we may assume that for some conjugate of $b$, and for  $\ge m/2$ values of $r$, the value $\lambda(p^g+r\tau(b^g))$ does not lie in any of the small `bad' disks to be removed.

\medskip

We now express our definable varieties relevant in the application of the theorem of Pila-Wilkie. For $\eta\in A$ and for $\xi\in T(\eta)$, we write $\xi=t_1\pi_1(\eta)+t_2\pi_2(\eta)$; we further write, as above, 
$\ell_0(z)=\beta_{1}(z)\omega_{1}(z)+\beta_{2}(z)\omega_{2}(z)$, where we now may omit the subscript ``$i$" since we use only the section $s_0=\sigma\circ \lambda$ and disregard $s_1$. Here $z$ lies in $F_\eta$. (Recall however that these functions depend - locally analytically - also on $\eta$.)

Let finally $\e=\e_\eta:T(\eta)\to F_\eta$ denote the elliptic-exponential. 

Our variety $Z$ is described in $A\times \R^4$ by the points $\eta\times (t_1,t_2, \beta_1(z),\beta_2(z))$, where $z=\e(t_1\pi_1(\eta)+t_2\pi_2(\eta))\in F_\eta$ and where $0\le t_1,t_2\le 1$ and $\eta\in A$.  We view this as a definable  family of varieties $Z_\eta$, as $\eta$ varies in $A$, and we shall be interested in a single $Z_\eta$, i.e. $Z_b$, where $b=\mu(p)$ (or possibly a conjugate of it, chosen according to what has been stated above).
 
\medskip

For $r=0,1,\ldots, m-1$, let $\xi_r$ denote a representative for $r\tau(b)$, lying in the above   fundamental domain $T(b)$ and let $\tilde\xi$ be a similar representative for $p$; recall that, by the above argument involving bounded height, we may assume that  for a set $R\subset\{0,1,\ldots ,m-1\}$ with $|R|\gg m$, $\tilde\xi+\xi_r$ lies in the double $2T^*(b)$ of the above modified  fundamental domain, for $r\in R$. We have $\xi_r=\theta_{1r}\pi_1(b)+\theta_{2r}\pi_2(b)$, where $\theta_{1r},\theta_{2r}$ are rational numbers in $[0,1]$ with denominator dividing $m$.\footnote{We have $\theta_{ir}\equiv r\theta_{1r}\pmod\Z$.} Also, letting $z_r=\e(\tilde\xi+\xi_r)$, we have that $\beta_1(z_r),\beta_2(z_r)$ are rationals with denominator at most $m^5$, since we are in the {\it Second case}. 

These  $m$ rational points are pairwise distinct, so applying again Theorem 1.9 of \cite{PW} (with any fixed $\epsilon <1/4$), we deduce that the algebraic part of $Z_b$ is not empty, which means that $Z_b$ contains an arc of real-algebraic curve in $\R^4$. Let $\xi$ be the above variable restricted to this arc,  so $t_1,t_2$ are algebraic functions of $\xi$. (Note that $t_1,t_2$ of course are determined as functions of $\xi$ anyway, but if $\xi$ is not restricted to the arc, these functions are not algebraic.)  The same holds for  $\beta_1(\e(\xi)), \beta_2(\e(\xi))$.  In other words, we may view $\xi$ as an elliptic logarithm (relative to $F_b$) of $z\in F_b$ and we have that $\ell_0(z)$ has ``Betti cooordinates" which are algebraic functions of $\xi$, when we restrict $\xi$ to a suitable real-analytic  arc inside $T(b)$.  Recall that here $\ell_0(z)$ is the elliptic logarithm of $s_0(z)$ relative to the elliptic curve $E_{\lambda(z)}$. 

Let us set $B_i(\xi):=\beta_i(\e(\xi))$, so $\ell_0(\e(\xi))=B_1(\xi)\omega_1(\e(\xi))+B_2(\xi)\omega_2(\e(\xi))$.  When $\xi$ runs on the said arc, the $B_i$  become equal to  certain algebraic functions of $\xi$, denoted $f_1(\xi), f_2(\xi)$.  So the  equality $$\ell_0(\e(\xi))=f_1(\xi)\omega_1(\e(\xi))+f_2(\xi)\omega_2(\e(\xi))$$ must hold on a whole disk (in fact on a whole neighborhood of the arc) by analytic  continuation; more precisely, it holds for a suitable branch of the algebraic functions in question, for some continuations of the functions $\omega_i(\e(\xi))$ to the said disk in the fundamental domain $T(b)$.\footnote{Note however that $(f_1,f_2)$ will not  have to be equal to the Betti map of $\ell_0$ on the whole disk.}

This kind of condition  seems not free   of independent  interest in the context of Betti maps. To contradict it we can use several arguments, based e.g. on monodromy. We choose the following one: consider the points  $w_i$  
in $F_b$ of bad reduction for $E_\lambda$, i.e. such that the fiber above $\lambda(w_i)$ is not elliptic; to each of them is associated a monodromy representation on the periods, obtained by first choosing a base point $y_0\in\P_1$,  travelling until near $w_i$, then travelling once around $w_i$ and coming back along the previous path, and looking at the transformation of the periods $\omega_1,\omega_2$. This  is represented in $SL_2(\Z)$.

Suppose now that $y_0$ is already near $w_i$. Let us pick 
a point $\xi \in T(b)$ representing $y_0$.  Let $M\in SL_2(\Z)$ be the matrix representing the above monodromy action around $w_i$. Now, if we travel with $\xi$ around a point  representing  $w_i$ on $T(b)$, say $h$ times, the algebraic functions $f_1,f_2$ will remain unchanged for suitable choice of $h>0$. 
This will correspond 
to $M^{h}$ acting on the column vector $\v=\v(\xi):=(\omega_1(y_0),\omega_2(y_0))^t$. On the other hand, 
the section $s_0\circ\lambda$ on $F_b$ will remain unchanged, hence $\ell_0$ will differ from the previous value by some linear combination $a_1\omega_1(y_0)+a_2\omega_2(y_0)$ with integer coefficients. 
Subtracting the expressions so obtained we get
\begin{equation}
(f_1(\xi),f_2(\xi)) \cdot (M^{h}-I) \v+(a_{1h},a_{2h})\cdot \v=0.
\end{equation}
This means that $(f_1(\xi),f_2(\xi)) \cdot M^h - (f_1(\xi)+a_1,f_2(\xi)+a_2)$ is orthogonal to $\v$.
This expression would hold for all $\xi$ in a disk by analytic continuation, where the integers $a_{1h},a_{2h}$ remain constant.  
Iterating this argument, i.e. travelling twice around $w_i$, we obtain an analogous relation with $2h$ in place of $h$.  
Then the two vectors
$$
(f_1,f_2)(M^h-I)+(a_{1h},a_{2h}), \quad (f_1,f_2)(M^{2h}-I)+(a_{1,2h},a_{2,2h})
$$
are parallel. Taking the wedge product leads us to the conclusion that
$$
(f_1,f_2)(M^h-I)\wedge (f_1,f_2)(M^{2h}-I)
$$
 is a linear form in $f_1,f_2$ with integer coefficients.
On the other hand, the above wedge product is a  quadratic form equal to $(\det(M^h-I)\det((f_1,f_2)^t, M^h(f_1,f_2)^t)$.
Iterating once again, we obtain another quadratic equation of that kind, with $M^2$ in place of $M$. 

Now, we may find a point $y_0$ so that $M$ is not unipotent and has not finite order (think e.g. of a Legendre model, for which the monodromy group is of finite index in $SL_2(\Z)$). Under this assumption on $M$, the two quadratic relations turn out to be independent; hence  the functions $f_1,f_2$ must be constant, hence real constant (since they are real on an arc). 
But the Betti image $(f_1(\xi),f_2(\xi))$ cannot be constant in view of Manin's theorem   (see \cite{Man} or  \cite{Z}).

\medskip

\subsection{Proof of Theorem \ref{T.BT}}


We can view the context as a scheme over an open  subset $\X_0$ of $\X$ with fibers given as products of three elliptic curves. Namely, above a point $x\in\X_0$ we associate the product of elliptic curves
$$
E_{\lambda(x)}\times F_{\mu(\tilde\sigma(x))}\times E_{\lambda(\tilde\tau\circ\tilde\sigma(x)))}.
$$
 For simplicity of notation, we denote by $\E_i\to \X_0$, $i=1,2,3$, the corresponding elliptic schemes,
where we now  define $\X_0$ just by removing from $\X$ the subset  where some of the relevant maps is not defined, or some elliptic curve degenerates. As before, $\X_0$ is the complement in $\X$ of a finite union of points and curves, namely the points where $\tilde\sigma$ or $\tilde\tau\circ\tilde\sigma$ is not defined, plus the union of the singular fibers corresponding to the maps which appear. 

We have an obvious section  $\xi$ given by  (with the above compressed notation) 
$\xi(p)=\sigma(p)\times \tau(\tilde\sigma(p))\times \sigma((\tilde\tau\circ\tilde\sigma)(p))$.

We start by proving the following 

\begin{lemma}\label{L.non-iso}
The above defined elliptic schemes are pairwise non-isogenous.
\end{lemma}
 
\begin{proof}
Let us denote by $j_\lambda:\X_0\to \C$ (resp. $j_\mu$) the $j$-function corresponding to the $\lambda$-scheme on $\X_0$ (resp. to the $\mu$-scheme). 
An isogeny between the schemes $\E_1$ and $\E_2$ would correspond to an algebraic relation of the form
$$
P(j_\lambda,j_\mu\circ \tilde{\sigma})=0
$$
for a non-zero polynomial $P(X,Y)\in \C[X,Y]$. But now note that  $j_\lambda$ is trivially invariant by $\tilde{\sigma}$, hence  $j_\lambda,j_\mu$ would  be  in fact algebraically dependent, which is excluded since the map $(\lambda,\mu)$, hence the map $(j_\lambda,j_\mu)$, is dominant. 

An isogeny between $\E_1$ and $\E_3$ would imply an algebraic dependence relation as before between $j_\lambda$ and $j_\lambda\circ\tilde{\tau}\circ\tilde{\sigma}$, which in turn would imply a dependence between $j_\lambda$ and $j_\lambda\circ \tilde{\tau}$ (again because $\tilde\sigma$ acts on the $\lambda$-fibers). In this case we exploit  the argument already used earlier in the paper: on a generic fiber for $\mu$, translation by $\tau$ would preserve the bad reduction of the $\lambda$-scheme (restricted to the said general fiber for $\mu$), which is a finite set. But this is impossible, since $\tau$ has infinite order. In an even simpler way, if we restrict to a given generic  fiber for $\lambda$, it suffices to note that $\tilde\tau$ does not send the fiber in another fibe; on the other hand,  if this were true,  this would happen for any iterate of $\tilde\tau$ and therefore the first scheme would be isotrivial.

Finally,  an isogeny between $\E_2$ and $\E_3$ would imply an algebraic dependence between $j_\mu\circ\tilde{\sigma}$  and $j_\lambda\circ \tilde{\tau}\circ\tilde{\sigma}$, hence between $j_\mu$ and $j_\lambda\circ \tilde\tau$. Noticing that $j_\mu$ is invariant under $\tilde{\tau}$, we would obtain as in the first case a dependence between $j_\lambda$ and $j_\mu$, which is excluded. 
\end{proof}

We now want to prove that the set where the above section $\xi$ assumes torsion values is degenerate, namely is not Zariski-dense in $\X_0$. By definition, this set, denoted now $\Ex_0$,  is $\Ex_0:=\Ex\cap \X_0$.

We use the common  procedure of counting, gong back to \cite{MZ}.

\medskip

To start with, certainly we have bounded height for the   $p\in \Ex_0$, by Lemma \ref{L.silverman}.

Also, i the section $\xi$  has torsion order $n$ at $p$, then the field of definition of  $p$ has degree $\gg n^c$, for some positive $c$ (we use here Lemma \ref{L.masser}) and therefore there are at least such a number of distinct conjugates of $p$.

\medskip

 Since we have bounded height $h(p)$, we may assume that again $\gg n^c$  such conjugates lie in a fixed compact set $\X_1$ inside $\X_0$. For this conclusion, covering first $\X$ by finitely many affine varieties,  it suffices to consider, in each of them,  a polynomial $Q$ in the affine coordinates vanishing on  the curves in $\X\setminus \X_0$ and evaluate this polynomial at $p$. We obtain a number $Q(p)$  of bounded height, hence if $Q(p)\neq 0$, a positive  proportion of its  conjugates  must have absolute value bounded below by a fixed positive number $\epsilon >0$, depending only on $\X$, $Q$, and the bound for the height. Hence these conjugates stay in the compact subset of $\X_0$ defined by $|Q(p)|\ge \epsilon$.
 
  \medskip
 
 Now, in turn we can cover $\X_1$ with  finitely many compact (bi)disks $D$ (i.e. analytically isomorphic to compact disks in $\C^2$) such that in each $D$ we have three well defined pairs of periods $\omega_{i1},\omega_{i2}$ for the schemes $\E_i$, $i=1,2,3$, and three well-defined elliptic  logarithms $\ell_i$ for the sections, hence equations
 \begin{equation}\label{E.periods}
 \ell_i(z)=\beta_{i1}(z)\omega_{i1}(z)+\beta_{i2}(z)\omega_{i2}(z),\qquad z\in D,\quad i=1,2,3,
 \end{equation}
 where the Betti values $\beta_{ij}(z)$ are real.  
 
 For $z\in D$, the sextuple of the $\beta_{ij}(z)$ describes a compact real definable set $Z$ in $\R^6$. 
 
 \medskip
 
 Since the covering of $\X_1$ by the disks $D$ is finite and may be chosen independently of $p$, we may also assume that $D$ too contains $\gg n^c$ conjugates of $p$.

Thus we obtain $\gg n^c$ rational points in $Z$.

\medskip

{\bf The algebraic part}.  If this holds for sufficiently large values of $n$, by Pila-Wilkie there is an algebraic arc  $C$ in $Z$, and we  let $C'$ be  its inverse image  in $D$; this $C'$ will be a real-analytic set of dimension $\ge 1$. 

 If we restrict to $C'$ the nine functions $\ell_i,\omega_{i1},\omega_{i2}$, $i=1,2,3$, the equations \eqref{E.periods} will provide at least two independent algebraic relations among the nine functions (more precisely, each six-tuple of functions, namely  for $i=1,2$, $i=1,3$ and $i=2,3$, will satisfy a nontrivial algebraic relation).

To exploit this fact, we apply a new result of Ax-Schanuel type, by Bekker-Tsimerman. We shall need only a special case, appearing as  Theorem 4.1 of \cite{BT}, which we restated as follows, in our notation as a lemma:

\begin{lemma}\label{L.Tsimerman}
Let $S$ be an algebraic variety, and for $i=1,\ldots,n$, let $\mathcal{E}_i\to S$, be non-isotrivial pairwise non-isogenous elliptic schemes, provided with sections $s_i:S\to\mathcal{E}_i$. For a simply-connected open set $D\subset S$ let, for $i=1,\ldots,n$, $l_i$  be an  elliptic logarithm of $s_i$  and $\omega_{i,1},\omega_{i,2}$  a basis for the periods of $\mathcal{E}_i$,  on $D$.    Let $F=(l_i,\omega_{i,1},\omega_{i,2})_{i=1\ldots,n}:D\to \C^{3n}$ be the corresponding map. Let $T\subset \C^{3n}$  be a codimension $k$ algebraic subvariety and suppose that $F^{-1}(T)$ contains an irreducible analytic component  $R$ of codimension $<k$. Then $R\neq S$ and either two of the elliptic schemes become isogenous on $R$, or at least two of the sections become torsion on $R$, or an elliptic scheme becomes isotrivial on $R$.
\end{lemma}

We apply this  lemma with $n=3$, $k=2$, $\E_1,\E_2,\E_3$ as defined above. By Lemma \ref{L.non-iso}, the three schemes $\E_i$, $i=1,2,3$ are pairwise non-isogenous, as requested in the above statement. 

 We take for $T$ the   algebraic variety defined in $\C^9$  by the said two independent relations,  while   $R$ will be a connected component of  the analytic closure in $D$ of  $C'$ (which is real-analytic, being  the inverse image through an analytic map of a real-analytic arc).

Note also that the two independent algebraic relations among the said nine functions must continue to be true in $R$, and note that $R$ has (complex) codimension $<2$ since $C$ has positive real-dimension.

We obtain from the lemma   that one of the three following alternatives hold:

(i) Two of the elliptic schemes become isogenous on $R$.

(ii) Two of the sections become torsion on $R$. 

(iii) One of the elliptic schemes becomes isotrivial on $R$.

In each of the three cases we obtain the important consequence  that $R$ is actually algebraic, and we may assume it is an algebraic curve, because otherwise the conclusion would extend to  the whole $\X$, which we have excluded since the beginning.


\medskip

Recall that still we know that each sextuple of    functions on $D$  given by $\ell_i,\omega_{i1},\omega_{i2}$, $i$ varying in any pair in $\{1,2,3\}$,  generate over $\C$ a field of transcendence degree at most $5$. Now, a (special case of) a result by Andr\'e (see Theorem 3, \S 7 of \cite{A}, the  case needed here having been proved before by Bertrand) implies  that, either one of the schemes becomes isotrivial on $R$, or, for any pair of them, either they are isogenous or one  of the section is torsion. 

Hence we have a somewhat stronger conclusion with respect to that of Lemma \ref{L.Tsimerman}.

Recall also that we are assuming that the variety $Z\subset \R^6$ given by the Betti coordinates of the three sections (restricted to $R$)  contains an algebraic arc. 

\medskip

Suppose now that two of the schemes $i=a,b$ are isogenous. On the said algebraic arc, the four functions $\ell_a,\omega_{1a},\omega_{2a}, \ell_b$ generate a field of transcendence degree at most $3$. By the theorem of Andr\'e we conclude that the two sections are linearly dependent (if we view them as section for a same  elliptic scheme, after applying the isogeny). Hence the two elliptic logarithms $\ell_a,\ell_b$ are linearly dependent over $R$ (up to the addition of periods). (Note that this is stronger than what is delivered by alternatives (i), (ii), (iii) above.)



If on the other hand the two schemes are not isogenous on $R$, then the same argument proves that one of the sections is torsion on $R$.

\medskip

Let us first work under the  assumption that {\bf none of the three schemes becomes isotrivial on $R$}. This implies that none of the sections is torsion on $R$. Indeed, if for instance the first section were torsion, then, since $\sigma$ is not torsion on $\X$, $R$ would be a fiber of $\lambda$, hence the first scheme would be constant. Similarly, if the second section is torsion on $R$, this means that $\mu\circ\tilde{\sigma} $ is constant on $R$, and the second scheme would be isotrivial. The same for the third one.

Hence, by the above, we may assume that any two of the schemes are isogenous and the corresponding sections are linearly dependent. 


Then we obtain that for $z\in R$ the curves $E_{\lambda(z)}$ and $F_{\mu(\tilde\sigma(z))}$ are isogenous, and so are the curves $E_{\lambda(z)}$ and $E_{\lambda(\tilde\tau\circ\tilde\sigma)(z))}$. For notational convenience it is better (and equivalent) to say that for $z$ on the curve $\tilde\sigma(R)$ the curves $E_{\lambda(z)}$, $F_{\mu(z)}$ and $E_{\lambda(\tilde\tau(z)}$ are isogenous.

We shall apply this fact to the first and third sections, namely with $a=1,b=3$; note that in this case the elliptic schemes derive from the same scheme on $\X$, and with the same sections, but evaluated at different points, namely $\lambda(z), \lambda(\tilde\tau(z))$.

We note that this result  represents a problem of Unlikely Intersections in the pure function field case, namely without an analogue in the number-field case; so {\it a fortiori}  we cannot say that this case comes from the number-field context.

To deal with this issue, we could argue again using the counting method and arguments of the type as in \cite{MZ}. However we can conclude directly by means of Theorem \ref{T.shioda}, as follows.

 
 In fact, Theorem \ref{T.shioda} allows to conclude the proof of Theorem \ref{T.BT}, even with no exceptional curves,  when none of the schemes is isotrivial on $R$.


If the second scheme is isotrivial then the above argument still works. 


If the first or third scheme is isotrivial on $R$,  we contend that the degree of the isogeny between the other two  is   bounded independently of   $R$ (note that indeed the other two must be isogenous since none of them can be isotrivial).  Indeed, these isogenies correspond to certain (modular) algebraic relations between the corresponding $j$-invariants of the relevant elliptic curves (these relations appearing also in the proof of Theorem \ref{T.shioda}). These $j$-invariants are rational functions of bounded degree of the coordinates of the point $p\in \X_0$: in fact, the elliptic curves depend rationally on $p,\tilde\sigma(p), (\tilde\tau\circ\tilde\sigma)(p)$.  If $p$ lies either  in a fiber of $\lambda$ or in a fiber of $\lambda\circ\tilde\tau\circ\tilde\sigma$, these degrees  are bounded independently of the fiber.  Therefore the isogeny degree is bounded, and the finiteness of the exceptional curves follows.


\subsection{Proof of Theorem \ref{T.shioda}}

After an isomorphism, and after taking a   finite degree (possibly ramified) cover of the base $B$, we can suppose that the scheme $\E$ is in Legendre form, so given by an affine equation $E_t:\, y^2=x(x-1)(x-t)$, with $B$ a cover of the $t$-line (deprived of $\{0,1,\infty\}$). (The curve $T$ will be replaced by a component of the lift of $T$ to the extended $B\times B$.) We may assume that the isogeny is cyclic and of degree $d >1$, and have to prove the impossibility of the said situation.

On iterating the isogeny, and without changing the base $B$ or the section, it is not difficult to see that we may assume that the isogeny  has degree larger than any prescribed number. 

We have a projection  from $B\times B$ to $\P_1\times\P_1$, giving the Legendre parameters, denoted $t,u$; these $t,u$ can be viewed, by restriction, as rational functions on $T$.  Since $E_{t(p)}, E_{u(p)}$ are isogenous for $p\in T$, the  curve $T$ projects in this way to one of the familiar modular curves in $\P_1\times\P_1$: one first uses the modular equation for the $j$-invariants  (see for instance \cite{L}), and then  just substitutes for $j$ the usual value in terms of the Legendre parameter.  

 So, the projection $T_0$ of $T$ to $\P_1\times\P_1$ is defined by taking the  said  irreducible  plane equation relating $t,u$ and $T$ will be  a suitable component of a lift of this curve to $B\times B$. The plane equation defining $T_0$ (belonging to the explicit family just described) is symmetric in $t,u$. Then   $(t,u)$ is  a $\C$-generic point of $T_0$, and $\C(T_0)\cong_\C\C(t,u)$. We denote by  $B_t,B_u$ the two copies of $B$, having two corresponding projections to the $t$-line or $u$-line. We thus obtain  two function field extensions $\C(B_t)/\C(t)$, $\C(B_u)/\C(u)$, which are isomorphic (over $\C$) with an isomorphism extending the natural isomorphism $\C(t)\cong_\C\C(u)$ sending  $t\to u$. 

We may view the sections as giving points $s_t\in E_t,s_u\in E_u$ of  Legendre elliptic curves $E_t, E_u$, the points being defined over $\C(B_t),\C(B_u)$ respectively.

 We note that  the extension $\C(t,u)/\C(t)$ (which corresponds to a  smooth model of $T_0$ with a map to $\P_1$) is   unramified outside $t\in S:=\{0,1,\infty\}$, and the places $t\in S$ of such function field extension  correspond exactly to the places $u\in S$: this is because of course $S$ is the exact locus of bad reduction of the Legendre curve, which is the same for isogenous curves.

Define $L$ as the field obtained by adding to $\C(t,u)$ the coordinates of all the $d$-torsion points of $E_t$ and $E_u$.  
The extension $L/\C(t)$ is finite,  Galois and also unramified outside $t\in S$, and similarly for $L/\C(u)$, by general theory.
Moreover, the isogeny $\phi$ is defined over $L$: in fact, the isogeny corresponds to taking a quotient of $E_t$ by a finite cyclic subgroup of order $d$. 
\medskip

Now we make a normalisation which  strictly is not needed, but will ease the arguments. We may replace the sections $s_t,s_u$ by a nonzero multiple $qs_t,qs_u$ without changing the assumptions. The respective  extension-fields of definitions $\C(B_t)/\C(t),\C(B_u)/\C(u)$ can only decrease, and,  if $q$ is sufficiently divisible, the fields  will stabilise, and thus we may suppose that both fields are minimal over all choices of $q$. After this normalisation is performed,  minimality implies that the field of definition of $bs_t$ is precisely $\C(B_t)$ for any positive integer $b$, and similarly for $u$ in place of $t$. 

\medskip

We let $L_t=L\cdot \C(B_t)$ and similarly for $u$ (here of course we view  all fields embedded in a fixed algebraic closure of $\C(t)$).   Note that by symmetry we have the equality of  degrees  $[L_t:L]=[L_u:L]$. Also, the equality $a\phi(s_t)=bs_u$, corresponding to our dependence assumption,  implies that $bs_u$ is defined over $L_t$. By the normalisation, $s_u$ is already defined over $L_t$ and  hence $L_t$ contains $L_u$, whence in fact $L_t=L_u$. 

\medskip

Now comes the crucial point.  Since the section $s_t$ is not torsion, by a theorem of {\sc Shioda}  the minimal  field of definition of the section $s_t$ (that is, $\C(B_t)$)  is ramified above at least one value of $t$ outside of $S:=\{0,1,\infty\}$; this is a somewhat  delicate theorem, which may be proved in different ways. See e.g. the paper \cite{CZ3} for a `modular' proof, a sharpening, and (in an appendix) a simplification of {\sc Shioda}'s proof.  

Let then $t_0$ be such a value. Then $t_0$ is unramified in $L$, but is ramified in $L_t$, and therefore it must be ramified in $L_u$. Take now a value $u_0\in\C$ such that the place $u=u_0$ lifts in $\C(t,u)$ to a place above $t=t_0$. Since $t=t_0$ is unramified in $\C(t,u)$, there are  $d:=[\C(t,u):\C(t)]$ such values. Such a place $u=u_0$ of $\C(u)$ is unramified in $L$ as well (because $L/\C(u)$ is ramified only above $u\in S$). Therefore the place, being ramified in $L_t=L_u$,  must be ramified in $\C(B_u)/\C(u)$. But the number of ramified places in this extension depends only on the section, whereas we can choose $d$ as large as we please. This contradiction proves finally the result.



\medskip

\begin{remark} \label{R.newproof} A shorter proof of Theorem \ref{T.shioda} can be given on using the sharpening of Shioda's theorem obtained in \cite{CZ3} as Theorem 2.1. This says the following. We are given an elliptic scheme over $B$ as in Thm. \ref{T.shioda} above. 
   Let $\H\to B$ be the universal cover of $B=B(\C)$, which gives by pull-back an elliptic family over $\H$. The fundamental group of $B$ acts on $\H$ by automorphisms (i.e. is represented in $PSL_2(\R)$). On $\H$ one can define a basis of the periods of the said family, as well as a logarithm of any given section of the original elliptic scheme. The fundamental group then acts on the periods and on these logarithms; for instance it sends a given logarithm into a translate of  itself by a period. Theorem 2.1 of \cite{CZ3} says that the subgroup leaving the periods fixed, operates on  a given logarithm (of a non-torsion section) with  an  orbit which is a rank 2 sublattice of the lattice of periods.

Let us now see in short the essentials of this different argument for the present theorem. For $z\in \H$, let $\tau(z)=\omega_2(z)/\omega_1(z)$ be the ratio  of a given basis of periods. We have a universal  cover of $B\times B$ given by $\H\times\H$. For coprime  integers $a,b,c,d$ with $ad-bc>0$, we have also a curve $\V=\V_{a,b,c,d}\subset \H\times\H$ defined by the pairs of points in $\H\times\H$  that correspond to a given cyclic  isogeny, i.e. $\V=\{(z_1,z_2)\in \H\times\H: \tau(z_2)=(a\tau(z_1)+b)/(c\tau(z_1)+d)\}$.   If we fix $\delta:=ad-bc$ to be the degree of our isogeny, this curve covers the curve $T$ inside $B\times B$  (corresponding to the cyclic isogenies of degree $\delta$): this is because $SL_2(\Z)$ operates transitively on the matrices of determinant $n$ (but this is not really needed). 

Note that $\V$ is sent to itself by the above mentioned subgroup (acting diagonally as $g(z_1,z_2)=(gz_1,gz_2)$), call it $G$,  fixing the periods (a fact which is essential). Indeed, $G$ fixes the periods, hence fixes $\tau$.

Let us denote by $\ell(z)$ a logarithm of our non-torsion section, well-defined in $\H$. The isogeny relating the elliptic curve over $z_2$ with that over $z_1$ is represented by multiplication by a $\xi=\xi(z_1,z_2)\in\C$ such that $\xi\omega_1(z_2)=a\omega_1(z_1)+b\omega_2(z_1)$,  $\xi\omega_2(z_2)=c\omega_1(z_1)+d\omega_2(z_1)$. 

Note that $\xi(gz_1,gz_2)=\xi(z_1,z_2)$  for $g\in G$, since $G$ fixes the periods. 

Our assumption on the dependence of the sections means that 
 for $(z_1,z_2)\in\V$,  we have $n\xi \ell(z_2)=m\ell(z_1)+\eta(z_1,z_2)$, where $\eta$ is a certain period (for the curve corresponding to $z_1$) and $m,n$ are nonzero integers.  Let $\omega_1,\omega_2$ be as above a basis of periods, viewed as functions well-defined on $\H$. Then, acting on the last equation with $G$, and using that $G$ acts as $\Z^2$  (and that it sends $\V$ into itself and fixes $\xi$) we get, for a suitable integer $h>0$ (which can be taken as the index of $G$ in $\Z^2$), 
\begin{equation}\label{E.shiodagroup}
n\xi (\ell(z_2)+h\omega_i(z_2))=m(\ell(z_1)+h\omega_i(z_1))+\eta(z_1,z_2),\qquad i=1,2.
\end{equation}
Hence, after subtracting the former equation, we obtain  $nh\xi \omega_i(z_2)=mh\omega_i(z_1)$, and so $\tau(z_1)=\tau(z_2)$ (for the present choice of basis of periods), which says that the isogeny has degree $1$.

\smallskip

It seems not easy, if at all possible,  to formulate this argument directly  in purely algebraic terms of Differential Galois Theory.
\end{remark}

\subsection{Proof of Theorem \ref{T.new1}}

We shall use an analogue of Silverman height bound Lemma \ref{L.silverman}, which was recently obtained by Yuan-Zhang \cite{YZ}. We quote just a corollary of their Theorem 6.2.2. Here is the statement:

\begin{lemma}\label{L.YZ} 
Let $A\to \X$ be an abelian scheme over an algebraic surface $\X$; let $\sigma_1,\sigma_2:\X\to A$ be independent sections. There exists a non-empty Zariski-open subset $\X_0\subset \X$ such that the  points of $\X_0$ where both sections take torsion values have uniformly bounded height.
\end{lemma}

Let 
\begin{equation*}
 \T=\{x\in\X: \sigma_i(x)\ \hbox{is torsion on $\E_i$ for $i=1,2,3$}\}.
\end{equation*}

By the lemma, there is a non-empty Zariski-open subset $\X_0\subset \X$ such that the points in $\T_0:=\T\cap\X_0$ have bounded height, hence in particular their degree over a number field of definition $K$ tends to infinity. 
Let then $x\in \T_0$ be a point such that its degree $d=[K(x):K]$ over $K$ is sufficiently large.

By bounded height, as in the proofs above in this paper, we may find a given compact set $\Delta\subset \X_0(\C)$, independent of $x$, such that at least $d/2$ conjugates of $x$ lie in $\Delta$. 

We may now partition $\Delta$ in a large but finite number of compact subsets $\Delta_j$, where we may assume that the periods related to the three schemes, and the logarithms of the sections, are well-defined in each $\Delta_j$. (Alternatively, we may assume since the beginning that $\Delta$ is simply-connected, on removing finitely many real-analytic threefolds.)
Whatever choice we make, we may assume that there is a compact $\Delta'$ independent of $x$, where the above functions are well-defined, and containing $>cd$ conjugates of $x$, where $c>0$ is independent of $x$.

\medskip

As before, we may write, for $\z\in \Delta'$,
\begin{equation} 
\ell_i(\z)=\beta_{1i}(\z)\omega_{1i}(\z)+\beta_{2i}(\z)\omega_{2i}(\z),\qquad i=1,2,3,
\end{equation}
where $\omega_{ij}$ are the periods and $\ell_i$ are the elliptic logarithms of the sections. 

\medskip

For $\z=x\in\T_0$, the $\beta_{ij}(\z)$ are rational numbers. 

Again by bounded height, Lemma \ref{L.masser} yields that the orders of torsion are bounded by $\ll d^2$.  So the denominators of the said rationals are likewise bounded. 

We apply the Pila-Wilkie estimates to the definable variety defined in $\R^6$ as the image $Z$ of the map $\z\mapsto (\beta_{ij}(\z))_{i=1,2,3, j=1,2}$ on $\Delta'$. Comparing all the bounds, we deduce that if $d$ is large enough, the algebraic part of $Z$ is non-empty, namely $Z$ contains a real algebraic arc.

Let us consider its pre-image $C'$ in $\Delta'$, which is a real-analytic set of dimension $\geq 1$. As in the previous proof, if we restrict to $C'$ the nine functions  $\ell_i, \omega_{i,1},\omega_{i,2}$ for $i=1,2,3$, the equations \ref{E.betti} will provide at least two independent algebraic equations among these functions (more precisely, each six-tuple of functions, namely  for $i=1,2$, $i=1,3$ and $i=2,3$, will satisfy a nontrivial algebraic relation).  We apply again Lemma \ref{L.Tsimerman} which provides the existence of a complex-analytic curve $R\subset \X$ containing $C'$ such that one of the following alternatives holds on $R$: 

- two of the elliptic schemes become isogenous on $R$, or 

- or  an elliptic scheme becomes isotrivial on $R$, or

- at least two of the sections become torsion on $R$

 In particular, we deduce that $R$ is an algebraic curve.

\medskip

Suppose that two of the schemes, say the first two of them, are isogenous but not isotrivial on $R$ and that at most one section is torsion (so the second and third alternatives do not hold). Then consider the pair formed by the first and third schemes. Associated to these schemes, we have the four periods, the two logarithms and the corresponding Betti maps. As before, we deduce that these six functions are algebraically dependent on $R$. But by the theorems of Andr\'e, this implies that these   schemes are isogenous, hence all the three schemes become isogenous on $R$. (Then among the five analytic functions consisting on the two periods and the three logarithms, at most three of them are algebraically independent.) 
Then we use Pila's result on the  Andr\'e-Oort Conjecture for $\mathbf{C}^n$  (in fact the functional part of the proof would suffice).   We argue as follows.  The three $j$-invariants $(j_1,j_2,j_3)$ define a rational map $J:\X\to \A^3$, and the closure of the image is a  surface in $\A^3$. Since the schemes become isogenous on $R$, the points of $J(R)$ with all three  CM coordinates are Zariski-dense in $J(R)$. Then the image $J(R)$ must be a  fiber-product of two modular curves by Pila's result. (In fact,  otherwise two of schemes would be identically isogenous on $\X$.) 
But the degree of the said surface is bunded and it follows that the degrees of these modular curves are also bounded. Hence there are only finitely many possibilities for them, and the same holds for $R$.
\footnote{The recourse to this case of Andr\'e-Ooort could avoided, for instance on using the piece of Pila-s proof dealing with the algebraic part  of the relevant definable set.} But then removing from $\X_0$ the finitely many curves $R$ which arise, we would conclude by the previous arguments.


\medskip

Let us now consider the case when the first  of the elliptic schemes becomes isotrivial on $R$.  So $R$ is a component of a curve of the shape $j_1(x)=c$, where $j_1$ is the $j$-invariant of the first scheme, viewed as a rational function on the base $\X$. Therefore $R$ has bounded degree (with respect to a  given embedding of $\X$). 

The six functions  $\ell_i, \omega_{i,1},\omega_{i,2}$ for $i=2,3$, become algebraically dependent on $C'$, hence on $R$. Again by Andr\'e-Bertrand, we deduce that either two of the schemes are isogenous or one of the section $\sigma_2$ or $\sigma_3$ is torsion on $R$. But this leads to at most finitely many curves $R$ because if the isogeny degree or the torsion order surpass a certain limit then the degree of the relevant curve grows.

\smallskip

Let us finally consider the case when none of the schemes is isotrivial on $R$ and $\sigma_1,\sigma_2$ become torsion on $R$. 
 
We  can then apply the main theorem of \cite{CMZ} which provides the finiteness of such curves.

\medskip

But this is only a sample and much more general things should follow with the same method.  

\appendix

\section{Relative Manin Mumford}
\subsection{The conjecture over $\Qbar$ and $\C$}

For a relative abelian scheme $\pi:A\rightarrow S$, we say that a subscheme $V\subset A$ is a torsion variety - or simply, torsion - if $V\subset A[N]$ for some positive integer $N$. If $V$ is reduced, this is equivalent to asking that for all $v\in V$, the point $v\in A_{\pi(v)}$ is torsion.
We denote by $Y_{tor}$ the union of all torsion subvarieties of $Y$.

If $B\subset A$ is an abelian subvariety, then for any positive integer $N$,  we say that any irreducible component of $B+A[N]$ is a torsion coset. 

The following is known as the \emph{Relative Manin-Mumford conjecture (RMM).}

\begin{conj}[RMM]\label{conj:RMM}
Let $\pi_0:A_0\ra S_0$ be a relative Abelian scheme, where $S_0$ is an  irreducible complex variety.  Let $Y_0\subset A_0$ be a subvariety with a Zariski-dense set of torsion points, such that $\dim Y_0+\dim S_0 <\dim A_0$. Then either $Y_0$ doesn't dominate $S_0$, or else its contained in a proper torsion coset. 
\end{conj}

We note that RMM has been recently announced by Gao-Habbeger over $\Qbar$. We wish to explain how to generalize this to $\C$. Our specific theorem doesn't rely on their methods but instead is a formal implication:

\begin{thm}\label{thm:RMM}
The RMM over $\C$ is a consequence of the RMM over $\Qbar$. 

\end{thm}

To prove this thoerem, we shall borrow ideas from the recent work\cite{BCFN} on Ax-Schanuel, though we do not use their specific theorems. 

We will in fact consider the following generalized version, with the case of $r=0$ being the RMM:

\begin{conj}\label{conj:main}
Let $\pi:A\ra S$ be an Abelian scheme over an irreducible complex variety $S$ and $Y\subset A$ be irreducible and dominate $S$. Let $Y_{tor}^{\geq r}$ denote the set of irreducible components of $Y_{tor}$ of dimension at least $r$. Suppose that $r+\dim A > \dim Y+\dim S$. If $Y_{tor}^{\geq r}$ is Zariski dense in $Y$, then $Y$ is contained in a proper torsion coset.
\end{conj}

Note that this is a special case of the zilber-pink conjecture, and is in fact equivalent to ZP restricted to a special case of weakly special subvarieties (those with no abelian part).

We will prove the following implication:

\begin{thm}\label{thm:main}

Conjecture \ref{conj:main} is a consequence of the RMM over $\Qbar$.

\end{thm}

We begin by showing that it is sufficient to consider our main conjecture for $\Qbar$ varieties - which is our main motivation for generalizaing to this case in the first place. This idea is already present in \cite{CMZ}.

\begin{lemma}
 Conjecture \ref{conj:main} over $\C$ follows from the $\Qbar$ case.
\end{lemma}

\begin{proof}

To prove the reduction, we spread out to $\Qbar$-varieties $\pi':A'\ra S'\ra T$ with $Y'\subset A'$ whose generic point over $T$ gives $\pi$. Now suppose that the theorem is false. Then there exist a zariski-dense set of varieties $s_i$ inside $Y_{tor}^{\geq r}$. Now these spread out to get a set varieties $S_i\subset Y$ whose fibers over the generic points of $T$ are the $s_i$, and such that $S_i\ra T$ is dominant and quasi-finite. Note that the dimensions of $S',A',S_i'$ are all $\dim T$ higher than what they were for their un-primed analogues. Thus replacing $r$ by $r+\dim T$ and applying the $\Qbar$ result yields that $Y'$ is contained in a proper torsion coset of $A'$, and taking the fiber over the genetic point of $T$ yields the same for $Y$.

\end{proof}

Henceforth we focus on proving Theorem \ref{thm:main} for $A,S,Y$ defined over $\Qbar$. WLOG we may assume that $A$ is principally polarized, as the problem is invariant under etale base change and isogenies.

\subsection{A finite-dimensional family of leaves}

Let $L=R^1\pi_*\Z$. Then $L$ is a local system over $S^{\an}$ underlying a variation of Hodge structures of weight 1. Now corresponding to the principl polarization on $A$ there is a mixed variation $E$ of over $A$ of weights $0,1$ which fits into an exact sequence $$0\ra \pi^*L^\vee\ra E\ra \Z(0)\ra 0.$$ This can be constructed as follows: We note that $\underline{Ext}^1(A,\bG_m)\cong A^\vee\cong A$ and so there is a `universal' group scheme $\tau:G\ra A$ which is an extension of  $A$ by $\G_m$, considered as constant group schemes over $A$. We take $E:=R^1\tau_*\Z(0)$.

Within $E_\cO$ there is a sub-bundle $F^0E$ which contains $\pi^*F^0L^\vee$ and also a lift of $\C(0)\otimes_\C\cO$. For a point $a\in A$ , if we let $e$ be a lift of a generator of $\Z(0)_a$ to $E_{a,\Z}$, then $e+l\in F^0E$ for $l\in \pi^*L^\vee$ implies that $l$ maps to $a$ under the natural universal covering map $L^{\vee,\an}\ra A^{\an}$. We let $E_1\subset E$ denote the closed subvariety of elements which map to $1\in\Z(0)$. 

Now let $\phi:E\ra A$ denote the map from the total space of $E$ and consider the splitting $\phi^*TA\xrightarrow{\sigma} TE$ given by the connection, locally cutting out the locus of a flat vector. Now we consider $$\cF_E:=\big(\sigma(\phi^*(TA))\cap TF^0E\big)\mid E_1.$$ In other words, $\cF_E$ is the tangent directions in $E_1$ along which some vector stays flat, and stays in $F^0$. By the description above, this means that once we fix the vector it determines its image in $A$. In other words, $\cF_E$ is integrable and yields a foliation of $E_1$ whose leaves are locally covering spaces over $S$. We shall be interested in the images of these leaves in $A$. 

In local co-ordinates, let $\ell\subset L^{\vee}$ be a flat section, which is to say a locally constant section of the first relative homology of $A$ over $S$. Then $\ell$ has an image in $A$ and this image is also the image of  the leaf in $\cF_E$ corresponding to $e+l$.

Importantly for us, for each $N$ all the components of $A[N]$ are images of leaves corresponding to the vectors in $\frac1NE_{1,\Z}$.

We call an image in $A$ of a leaf from $\cF_E$ a \emph{homology leaf}. 

\subsection{The foliation is unlikely everywhere}

\begin{lemma}\label{lem:everypoint?}
There is a Zariski-open subset $U\subset Y$ such that all points of $U$ have a homology leaf through them which intersects $Y$ in dimension at least $r$.
\end{lemma}

\begin{proof}

The set of points in question can be described algebraically. Note that while it is more intuitive to think of $F$ as a germ, the dimension condition can be checked at the level of formal schemes and is therefore algebraic, and thus the locus is constructible by Chevalleys theorem . Now this includes all of $Y_{tor}^{\geq r}$, which by assumption is Zariski-dense in $Y$. Any Zariski-dense constructible set must contain a Zariski open subset, and thus the result follows.

\end{proof}

\subsection{Unlikely Intersections in Mixed Shimura Varieties}

Let $\bA_g$ denote the universal principally polarized Abelian variety of dimenion $g$. This sits over $\cA_g$ - the moduli space of principally polarized abelian varieties. Now we have the following convenient lemma from \cite[Proposition 1.1]{Gao1}

\begin{lemma}\label{weaklyspecials}

Let $W\subset \bA_g$ denote a weakly special subvariety. Then $W$ lies a weakly special $W'\subset \cA_g$ and there exists a decomposition $\bA_g\mid W \sim A\oplus B\oplus C$ such that $C$ is isotrivial, and $W =A + (0,b,c)$ where $b\in B(\ol{W'})$ is a torsion section and $c\in C(W')$ is a constant section. 

\end{lemma}

Now, we apply the mixed Ax-Schanuel theorem of \cite{GaoAS}. Let $W$ be the weakly special closure of the image of $A$ in $\cA_g$ with image $W'\subset \cA_g$, and note that the Hodge leaves are pullbacks of algebraic sets from the universal cover of $\cA_g$. Then \cite[Thm 1.1]{GaoAS} tells us that at every point $y\in Y$ and Hodge leaf $L$ which intersects $Y$ in dimension at least $r$, there exists a weakly special $V\subset W$ whose pre-image in $A$ is a subvariety $R\supset Y\cap L$ with image $R'\subset S$, such that 
\begin{equation}\label{ASrelation}
\dim Y\cap R + \dim V'\geq \dim V + r.
\end{equation}

Since $\dim Y\cap R \leq \dim Y$ it must follows that $\dim V'-\dim V< \dim W'-\dim W=\dim A$, and thus by \ref{weaklyspecials} that $A\mid V'$ must decompose as $B+C+D$ where $D$ is isotrivial and $V\subset B+c+d$ where $c$ is torsion, such that $B$ a proper subvariety of $A\mid V'$. 

Now, since this decomposition occurs at every point of $y\in U$, and since there are only countably many decompositions of an abelian variety, it follows that there is a global decomposition $A = B+C+D$ whose restrictions to the $V'$ give the decomposition above. Now, since we only obtain torsion points of $C$, torsion points can't move continuuosly, and $Y$ is not contained in a proper weakly special of $A$, it follows that $C=0$. That means there exists a global decomposition $A=B+D$ where $D$ is isotrivial along the $V'$ fibers , and all the intersections lie over a constant section of $D\mid V'$.

We now complete the proof of Theorem \ref{thm:main}. Let $\psi:S\ra T$ denote the quotient of $S$ whose fibers are the $V'$, so that $D$ is the basechange of $E\ra T$ along $\psi$. Consider $Z
:=\ol{\psi(Y)}$.  Note that torsion of $Z$ must be Zariski dense in $Z$ and thus by the RMM over $\Qbar$ it follows that


\begin{equation}\label{Dineq}
\dim Z + \dim T \geq \dim E. 
\end{equation}

Now let $d\in E$ be the image of $y\in Y$. By the description above it follows that $R$ lies entirely within the preimage of $d$, and thus that $Y\cap R\subset Y_d$. Thuus we obtain
\begin{align*}
\dim Y\cap R-\dim V+ \dim V'&\leq \dim Y_d -\dim (B/S)\\
&=\dim Y-\dim Z-\dim (B/S)\\
&\leq \dim Y-\dim E + \dim T -\dim (B/S)\\
&=\dim Y -\dim (A/S)\\
&< r.
\end{align*}
contradicting \eqref{ASrelation}. This completes the proof.

\bigskip

\noindent{\bf Acknowledgments}. The authors are very grateful to Serge Cantat and Romain Dujardin for helpful discussions on the topics of the present work and for sending them a first version of their inspiring paper \cite{CD}.
We also thank Philipp Habegger for an e-amil exchange and for sharing with us a related recent preprint with Z. Gao.


\bigskip

\address{
Pietro Corvaja\\
Dipartimento di Scienze Matematiche e Informatiche e Fisiche\\
Universit\`a di Udine\\
Via delle Scienze, 206\\
33100 Udine  \\
Italy
}
{pietro.corvaja@dimi.uniud.it}
 
 \address{
Jacob Tsimerman\\
Department of Mathematics\\
University of Toronto\\
Canada
}
{jacobt@math.toronto.edu}

\address{
Umberto Zannier\\
Scuola Normale Superiore\\
Piazza dei Cavalieri, 7\\
56126 Pisa \\
Italy
}
{umberto.zannier@sns.it}


\begin{thebibliography}{99}

\bibitem{A}  Y. Andr\'e,  Mumford-Tate groups of mixed Hodge structures and the theorem of the fixed part, Compos. Math. 82( 1992),  1--24.
\smallskip

\bibitem{BT}  B. Bakker, J. Tsimerman, Functional Transcendence of Period Integrals, {\tt https://arxiv.org/abs/2208.05182}, (2022).
\smallskip

\bibitem{BM} E. Bierstorne, P. Milman, Semianalytic and Subanalytic Sets, {\it Publications Math. I.H.E.S.} (1998).

\smallskip

\bibitem{BG}  E. Bombieri, W. Gubler, Heights in Diophantine Geometry, Cambridge Univ, Press, 2006.



\smallskip


\bibitem{CD}   S. Cantat, R. Dujardin, Finite Orbits for Large Groups of Automorphisms of Projective Surfaces, 
{\tt{https://arxiv.org/abs/2012.01762}}, (2020).
\smallskip

\bibitem{chiu}
K.Chiu, Ax--{S}chanuel for variation of mixed Hodge structures,  
{\tt http://arxiv.org/pdf/2101.10968}
{\tt {arXiv:2101.10968}}, 2021

\smallskip

\bibitem{CMZ}  P. Corvaja, D. Masser, U. Zannier, Torsion hypersurfaces on abelian schemes and Betti coordinates, {\it Math. Annalen} {\bf 371} (2018), 1013-1045.\smallskip

\bibitem{CZ} P. Corvaja, U. Zannier, Applications of Diophantine Approximation to Integral Points and Transcendence, Cambridge Tracts in Math. n. 212, Cambridge Univ. Press, 2018.

\smallskip

\bibitem{CZ2}  P. Corvaja, U. Zannier, On the Hilbert Property and fundamental group of algebraic varieties, Math. Zeit. {\bf 286} (2017) no. 1-2,  579-602.

\smallskip

\bibitem{CZ3} P. Corvaja, U. Zannier, Unramified sections of the Legendre scheme and modular forms, J. of Geometry and Physics, {\bf 166} (2021), 26 pp.

\smallskip

\bibitem{CDMZ} P. Corvaja, D. Masser, J. Demeio, U. Zannier,  On the torsion values of sections of an abelian scheme, J. Reine Angew. Math. {\bf 782} (2022), 1-41.
\smallskip


\bibitem{David} S. David,  Fonctions th\^eta et points de torsion des vari\'et\'es ab\'eliennes, {\it Compositio Math.}, {\bf 78}, n. 2 (1991), 121-160.\smallskip

\bibitem{BCFN}
D. Bl\'azquez-Sanz, G.Casale, J.Freitag, J.Nagloo, A differential approach to the Ax-Schanuel, I, 

\smallskip

\bibitem{DimGaoHabegger} V. Dimitrov, Z. Gao, Ph. Habegger, Uniformity in Mordell-Lang for curves, Annals of Mathematics {\bf 194} (1) (2021), 237-298.

\smallskip

\bibitem{Gao1} Z.Gao, A special point problem of {A}ndr\'{e}-{P}ink-{Z}annier in the universal family of {A}belian varieties,
{\it Ann. Sc. Norm. Super. Pisa Cl. Sci.} {\bf 5} (2017), 231-266.
\smallskip

\bibitem{GaoAS} Z.Gao. Mixed {A}x-{S}chanuel for the universal abelian varieties and some applications, {\it Compositio Math.}
 {\bf 15} 2266 (2020)3--2297bitem{GrHa} P. Griffiths, J. Harris, A Poncelet theorem in space, {\it Comm. Math. Helvetici} {\bf 52} (1977), 145-160.
\smallskip

 

 \bibitem{HabAnnali} Ph. Habegger, Torsion points on elliptic curves in Weierstrass form, Annali Scuola Normale Superiore Cl. Sci (5) Vol. XII (2013), 1-29.  
 
 \smallskip

\bibitem{L} S. Lang,  Elliptic Functions, Graduate Texts in Mathematics 112, Springer, 1973.

\smallskip


\bibitem{Lang} S. Lang, Fundamentals of Diophantine Geometry, Springer 1983.\smallskip


\bibitem{Man} Yu. Manin, Rational points on Curves over function fields, {\it Izv. Ak. Nauk SSSR, Ser. Mat.} {\bf 27.6} (1963), 1395-1440. \smallskip


\bibitem{MZ} D. Masser, U. Zannier, Torsion anomalous points and familes of elliptic curves, {\it Am. J. Math} {\bf 132} (2010), 1677--1691.\smallskip

\bibitem{MZ1} D. Masser, U. Zannier, Torsion points on familes of squares of elliptic curves, {\it Math. Annalen} {\bf 352} (2012), 453-484.\smallskip

\bibitem{MZ2} D. Masser, U. Zannier, Torsion points on familes of products of elliptic curves,
{\it Advances in Math}. {\bf 259} (2014), 116-133.\smallskip

\bibitem{MZ3} D. Masser, U. Zannier, Torsion points on families of simple abelian varieties, Pell's equation and integration in elementary terms.

\bibitem{PW} J. Pila, A. Wilkie, The rational points of a definable set {\it Duke Math. J.} {\bf 133} (2005)  591--616.
\smallskip

 
\smallskip




\smallskip

\bibitem{YZ} X. Yuan, S. Zhang,  Adelic line bundles over quasi projective varieties, {\tt https://arxiv.org/abs/2105.13587}, (v4 2022).

\smallskip
\bibitem{ZEns}  U. Zannier, A local-global principle for norms from cyclic extensions of $Q(t)$ (a direct, constructive and quantitative approach), {\it L'Enseignement Math\'ematique} {\bf 45} (1999), 357-377.
\smallskip


 \bibitem{Z} U. Zannier, Some Problems of Unlikely Intersections in Arithmetic and Geometry, Princeton U. Press (2014).


\end{thebibliography}
\end{document}